\documentclass[11pt]{amsart}
\usepackage{stmaryrd}
\usepackage{graphicx} 
\usepackage{amsmath,amsthm,amssymb}
\usepackage{enumerate}
\usepackage{amsbsy}
\usepackage{dsfont}
\usepackage{color}
\usepackage{hyperref}
\usepackage{mathrsfs}

\newtheorem{theorem}{Theorem}[section]
\newtheorem{proposition}[theorem]{Proposition}
\newtheorem{lemma}[theorem]{Lemma}
\newtheorem{corollary}[theorem]{Corollary}

\theoremstyle{definition}
\newtheorem{definition}[theorem]{Definition}

\begin{document}

\title[Deformation of Contracting Maps under Heat Flow]{{Deformation of Contracting Maps under the Harmonic Map Heat Flow}}

\author[J.-L. Hsu]{Jia-Lin Hsu}
\address{Department of Mathematics, National Taiwan University, Taipei, 106, Taiwan}
\email{R12221007@ntu.edu.tw}

\author[M.-P. Tsui]{Mao-Pei Tsui}
\address{Department of Mathematics, National Taiwan University, Taipei, 106, Taiwan;
National Center for Theoretical Sciences, Mathematics Division, Taipei, 106, Taiwan}
\email{maopei@math.ntu.edu.tw}


\begin{abstract}
We investigate the homotopy classes of maps between closed manifolds by studying certain contracting conditions on the singular values of the differential of the map. Building upon the work of Lee and Wan \cite{lee2023rigiditycontractingmapusing}, we extend their results on 2-nonnegative maps to a more general class of k-nonnegative maps by exploiting the properties of submatrices and the linearity of the harmonic map heat flow. Our work establishes new rigidity theorems for maps between manifolds with specific curvature bounds and yields new homotopy rigidity results for maps between spheres and complex projective spaces.

\end{abstract}
\maketitle

\section{Introduction}

A fundamental question in geometry is to understand the relationship between the geometric properties of a map and its topological properties, such as its homotopy class. In recent years, there has been significant progress in this area, particularly in the study of maps between manifolds with specific curvature conditions.

For instance, Tsui-Wang \cite{MR2053760} proved that area-decreasing maps between spheres are homotopically trivial, a result that has been extended to maps between complex projective spaces by Tsai-Tsui-Wang \cite{tsai2023newmonotonequantitymean}. More recently, Lee-Tam-Wan \cite{lee2024rigidityareanonincreasingmaps} showed that area-nonincreasing self-maps of higher-dimensional complex projective spaces are either isometric or homotopically trivial. These results often rely on techniques involving mean curvature flow and the analysis of the evolution equations.

Lee-Wan \cite{lee2023rigiditycontractingmapusing}  introduced a new class of contracting maps, which lie between area-nonincreasing and distance-nonincreasing maps. They used the harmonic map heat flow to study the rigidity properties of these maps, proving that under certain curvature conditions, such maps are either Riemannian submersions or homotopically trivial.
One of their results is as follows:
\begin{theorem}[\cite{lee2023rigiditycontractingmapusing}, Theorem 1.2]
    Let $(M^m,g)$ and $(N^n,h)$ be closed connected manifolds. Suppose
    \begin{enumerate}
        \item the sectional curvature $K_{N}$ is positive;
        \item $\mathrm{Ric}_{M}(v)\geq \mathrm{Ric}_N(w)$ for any unit tangent vectors $v$ and $w$ in $M$ and $N$, respectively.
    \end{enumerate}
    Then, for every smooth map $F: M\to N$ such that $g-F^*h$ is $2$-nonnegative, $F$ is either a Riemannian submersion or homotopically trivial. In particular, $N$ is an Einstein manifold in the first case.
\end{theorem}
Lee and Wan \cite{lee2023rigiditycontractingmapusing}  demonstrated that this contracting condition is preserved under the harmonic map heat flow when certain curvature conditions are met. They then characterized the limiting maps by analyzing the evolution inequalities. A crucial step in their work was establishing the preservation of the contracting condition through the linearity of the harmonic map heat flow, which allowed them to apply properties of the curvature submatrix and address problems involving two singular values.

By systematically analyzing how Lee and Wan utilized the properties of submatrices in their work, we extend their approach to encompass contracting conditions involving an arbitrary number $k$ of singular values.
\begin{theorem}\label{k singular values, pinching}
    Let $(M^m,g)$ and $(N^n,h)$ be closed manifolds, $3\leq k\leq n$. Suppose
    \begin{enumerate}
        \item the sectional curvature $K_{N}$ is positive;
        \item $\mathrm{Ric}_{M}(v)\geq \mathrm{Ric}_N(w)$ for any unit tangent vectors $v$ and $w$ in $M$ and $N$, respectively.
    \end{enumerate}
    Denote the curvature pinching $\sup K_N/\inf K_N$ by $\kappa\geq 1$. Then there exists a constant $2/k<\phi_k(\kappa)\leq 1$ such that 
    for every smooth map $F:M\to N$, if $\phi_k(\kappa)g-F^*h$ is $k$-nonnegative, then either $F$ is homotopically trivial, or 
    $\phi=1$ and $F$ is a Riemannian fiber bundle.
\end{theorem}
These are new preserved contracting conditions since constants $\phi_k(\kappa)$ are always greater than $2/k$. On the other hand, for the contracting condition involving $n$ singular values, we have the following result:
\begin{theorem}\label{Einstein pinching}
    Let $(M^m,g)$ and $(N^n,h)$ be a closed manifold. Suppose
    \begin{enumerate}
        \item $N$ is an Einstein manifold of positive sectional curvature and curvature pinching $2$;
        \item $\mathrm{Ric}_{M}(v)\geq \mathrm{Ric}_N(w)$ for any unit tangent vectors $v$ and $w$ in $M$ and $N$, respectively.
    \end{enumerate}
    Then, for every smooth map $F: M\to N$ such that $|dF|^2\leq n$, either $F$ is homotopically trivial, or $F$ is a Riemannian fiber bundle.
\end{theorem}

Therefore, for a smooth map $F:(S^m,g_{S^m})\to (S^n,g_{S^n})$ with $m\geq n$, if $|dF|^2\leq n$, then either $F$ an isometry or homotopically trivial. On the other hand, for maps between complex projective spaces, we have the following result:
\begin{corollary}\label{Main theorem: complex projective spaces}
Let $F:(\mathbb{CP}^m,g_{FS})\to (\mathbb{CP}^n,g_{FS})$, where $m\geq n\geq 2$. If one of the following holds
\begin{enumerate}[(i)]
    \item $\lambda_i^2+\lambda_j^2+\lambda_k^2\leq 2.5$ for every three singular values $\lambda_i,\lambda_i$;
    \item $|dF|^2\leq \frac{n+2}{3}$,
\end{enumerate}
then $F$ is homotopically trivial.
\end{corollary}

\textit{Acknowledgement:} The first named author would like to thank Professor Mu-Tao Wang, Professor Man-Chun Lee, and Jingbo Wan for valuable discussions. J.-L. Hsu and M.-P.~Tsui are supported in part by the National Science and Technology Council grants 112-2115-M-002 -015 -MY3.
\section{Preliminaries}

\subsection{A brief introduction to the harmonic map heat flow}
For any smooth map $F:(M^m,g)\to (N^n,h)$ between Riemannian manifolds, we define the \textbf{tension field} of $F$ to be the vector field $\tau(F):=\mathrm{tr}_{g}(\nabla dF)$ along $F$, which can be expressed as
\[
    \tau(F)^{\alpha}=F^{\alpha}_{ij}g^{ij}
\]
in local coordinates. Then we define a \textbf{harmonic map heat flow} $F_t:M\to N(t\in[0,T))$ evolving from $F$ to be a solution to
\[
    \frac{\partial F_{t}}{\partial t}=\tau(F_t),F_0=F.
\]
When $\tau(F)=0$, we say $F$ is a \textbf{harmonic map}. When $M$ and $N$ are closed manifolds, the short-time existence and uniqueness of the harmonic map heat flow have been proved in \cite{eells1964harmonic}, and we can define the maximally extended harmonic map heat flow evolved from $F_0$.

By the standard parabolic PDE, we have the following priori estimates:
\begin{proposition}[\cite{MR1428088}]\label{interior estimate}
For a harmonic map heat flow $F_{t}: M\to N(t\in [0,T))$ from a closed manifold $M$ into a Riemannian manifold with bounded curvature, if $|dF_{t}|$ is uniformly bounded, then the higher order derivatives are also uniformly bounded.
\end{proposition}

Thus, the long-time existence and subsequential convergence are guaranteed when $|dF_{t}|$ is uniformly bounded.

\begin{theorem}\label{long-time existence and convergence}
Let $M$ and $N$ be closed manifolds, $F_{t}:M\to N(t\in [0,T))$ a maximally extended harmonic map heat flow. If $|d F_{t}|$ is uniformly bounded, then $T=\infty$, and there is $t_{k}\to\infty$ such that $F_{t_{k}}$ smoothly converges to a harmonic map $F_{\infty}$.   
\begin{proof}
By Proposition \ref{interior estimate}, derivatives of any orders of $F$ are uniformly bounded. Suppose $T<\infty$. Then by Arzel\'a-Ascoli theorem,
\begin{align*}
    G_{n}:M\times [\frac{T}{2},T]&\to N(n\geq 2)\\
    (p,t)&\mapsto F(p,t-\frac{T}{n})
\end{align*}
subsequently converges to some harmonic map heat flow
\[
    G:M\times [\frac{T}{2},T]\to N,
\]
which is compatible to $G$ on $[\frac{T}{2},T)$. However, $F_{t}:M\times [0,T)\to N$ is maximally extended, and it leads to a contradiction. Thus, $T=\infty$.

To prove the subsequently convergence, we consider
\begin{align*}
    H_{n}:M\times [0,1]&\to N\\
    (p,s)&\mapsto F(p,n+s)
\end{align*}
are uniformly bounded. By Arzel\'a-Ascoli theorem, $F_{n}$ subsequently converges a harmonic map heat flow
\[
    H_{\infty}:M\times [0,1]\to N
\]
of constant energy, so $H_{0}$, which is a subsequential limit of $F_t$, is harmonic.
\end{proof}
\end{theorem}

Whether the subsequential convergence can be enhanced to convergence remains
unknown in general. This uniformity question has been studied in \cite{lin2014uniformity}, \cite{topping1997rigidity}, and \cite{huang2013regularity}. We also prove this uniformity when the limiting maps are constant.

\begin{proposition}\label{attraction principle}
Let $(M^m,g)$ be a closed manifold and $F_{t}:M\to N(t\in [0,\infty))$ a harmonic map heat flow with uniformly bounded energy density and subsequently smoothly converges to a constant map $F_{\infty}$. Then $\{F_{t}\}$ converges smoothly to $F_{\infty}$.
\begin{proof}

First, we prove that $F_{t}$ converges uniformly to $F_\infty$ by applying the maximum principle on $\rho_p^2\circ F_{t}$, where $\rho_{p}$ denotes the distance function to $p\in N$.
Let $p$ denote the image of $F_{\infty}$.
By direct computation, the evolution equation of $\rho_p^2\circ F_{t}$ is
\begin{align*}
    \left(\frac{\partial}{\partial t}-\Delta\right)\rho_p^2\circ F_{t}&=(\rho_p^2)_\alpha\frac{\partial F^\alpha}{\partial t}-\left( (\rho_p^2)_{\alpha}F^\alpha_i \right)_{j}g^{ij}\\
    &=(\rho^2_p)_\alpha\frac{\partial F^\alpha}{\partial t}-(\rho^2_p)_{\alpha\beta}F^\alpha_iF^\beta_jg^{ij}-(\rho_p^2)_\alpha F^{\alpha}_{ij}g^{ij}\\
    &=-(\rho^2_p)_{\alpha\beta}F^{\alpha}_{i}F^\beta_{j}g^{ij}.
\end{align*}
Since $(\rho^2_p)_{\alpha\beta}=2h_{\alpha\beta}$ at $p$, there is a geodesic ball $B_R(p)$ such that $(\rho^2_p)_{\alpha\beta}$ is positive definite on it, and we have
\[
    \left(\frac{\partial}{\partial t}-\Delta\right)\rho_p^2\circ F_{t}\leq 0.
\]
Therefore, $\sup\rho_p^2\circ F_t$ is monotone decreasing when $F_{t}(M)\subset B_R(p)$, and $F_{t}$ uniformly converges to $F_\infty$.
Consequently, the subsequential limit of $F_{t}$ is unique.

By Proposition \ref{interior estimate}, any sequence $t_k$ always admits a smoothly convergent $F_{t_{k_{\ell}}}$, so we derive the smooth convergence.
\end{proof}
\end{proposition}

When the limiting map $F_\infty$ has the same Dirichlet energy as the initial map $F$, the deformation of harmonic map heat flow can also be characterized explicitly.

\begin{proposition}\label{equal total energy}
Let $(M^m,g)$ and $(N^n,h)$ be closed connected manifolds, $F:M\to N$ a smooth map.
If the harmonic map heat flow $F_{t}:M\to N(t\in [0,\infty))$ subsequently smoothly converges to a harmonic map $F_{\infty}$ with $E(F_{\infty})=E(F)$, then $F=F_{\infty}$.
\begin{proof}
Since the harmonic map heat flow is the gradient descent of Dirichlet energy, for any harmonic map heat flow $F_{t}:M\to N(t\in [T_1,T_2])$, if $E(F_{T_1})=E(F_{T_2})$, then $F_{t}$ is time independent.

Now, let $t_{k}$ be a sequence increasing to $\infty$ such that $F_{t_k}$ smoothly converges to $F_{\infty}$. Then $E(F_{t_{k}})$ is a sequence decreasing to $E(F_{\infty})$ and bounded above by $E(F)=E(F_{\infty})$, and thus $E(F_{t_k})=E(F)$ for all $k$. 

According to the case of closed intervals, $F_{t_{k}}=F$. Consequently, $F=F_{\infty}$.
\end{proof}
\end{proposition}

\subsection{Evolution equations}

We aim to establish time-independent upper bounds for energy density to prove the long-time existence and convergence of harmonic map heat flow. Therefore, we are concerned with the evolution equation of pull-back metric and energy density.

The following evolution equation can be found in \cite{lee2024rigidityareanonincreasingmaps}.
\begin{proposition}\label{evolution equation of energy density}
For a harmonic map heat flow $F_{t}:(M^m,g)\to (N^n,h)(t\in [0,T))$,
\begin{align*}
    \left(\left(\frac{\partial}{\partial t}-\Delta\right)F^* h\right)_{ij}=&-(\mathrm{Ric}_M)_i^{\ell}(F^*h)_{\ell j}-(\mathrm{Ric}_M)_j^{\ell}(F^*h)_{\ell i}+2(F^*R_N)_{kij\ell}g^{k\ell}\\
    &-2g^{k\ell}F^{\alpha}_{ik}F^\beta_{j\ell}h_{\alpha\beta},
\end{align*}
therefore,
\[
    \left(\frac{\partial}{\partial t}-\Delta\right)e(F)=-|\nabla dF|^2-(\mathrm{Ric}_M)^{ij}F_i^\alpha F_j^\beta h_{\alpha\beta}+(F^*R_N)_{kij\ell}g^{k\ell}g^{ij}.
\]
\end{proposition}
To simplify the evolution equation and apply the tensor maximum principle, we need to consider the frame of singular value decomposition (SVD frame), first introduced by M.T. Wang in \cite{wang2002long} for studying graphical mean curvature flow.

Formally speaking, for a $\mathcal{C}^2$-map $F\colon(M^m,g)\to (N^n,h)$ between Riemannian manifolds, we say a \textbf{singular value decomposition frame (SVD frame)} is a pair of orthonormal frames of $TM$ and $F^*TN$ such that $F^{\alpha}_{i}=\lambda_i\delta^\alpha_i$ is diagonalized and $\lambda_1\geq\cdots\geq \lambda_m\geq 0$. For convenience, we also let $0=\lambda_{m+1}=\lambda_{m+2}=\cdots$. Be aware that a SVD frame is not necessary smooth.

\begin{corollary}\label{evolution equation written in singular values}
For a harmonic map heat flow $F_{t}:(M^m,g)\to (N^n,h)(t\in [0,T))$, with respect to a SVD frame,
\begin{align*}
    \left(\left(\frac{\partial}{\partial t}-\Delta\right)F^* h\right)_{ii}\leq &-2(\mathrm{Ric}_M)_{ii}\lambda_i^2+2\sum_{j=1}^{\min(n,m)}(K_N)_{ij}\lambda_i^2\lambda_j^2,
\end{align*}
where the equality holds at $(p,t)$ when $F$ is totally geodesic at it, and
\[
    \left(\frac{\partial}{\partial t}-\Delta\right)e(F)=-|\nabla dF|^2-\sum_{i=1}^{\min(n,m)}(\mathrm{Ric}_M)_{ii}\lambda_i^2+\sum_{i,j=1}^{\min(n,m)}(K_N)_{ij}\lambda_i^2\lambda_j^2
\]
\begin{proof}
It follows from Proposition \ref{evolution equation of energy density} directly.
\end{proof}
\end{corollary}

In this formula, the concept of partial Ricci curvatures naturally arises, which was first introduced by Z.M. Shen in \cite{shen1993complete} and H. Wu in \cite{wu1987manifolds}. We also introduce the positive part of partial Ricci curvature, which will be used later.

\begin{definition}
For a Riemannian manifold $N$ and orthonormal vectors $v_0,v_1,\ldots,v_q$ in some tangent space $T_pN$, we define the 
\textbf{partial $q$-Ricci curvature} to be
\[\mathrm{Ric}_N^q(v_0;v_1,\ldots,v_q):=\sum_{i=1}^{q}R_N(v_0,v_i,v_i,v_0)
\]
and the \textbf{positive part of partial $q$-Ricci curvature} to be
\[(\mathrm{Ric}_N^q)_{+}(v_0;v_1,\ldots,v_q):=\sum_{i=1}^{q}\max(0,R_N(v_0,v_i,v_i,v_0))
\]
For a subspace $W^{q+1}\leq T_pN$, we can also define the partial $q$-Ricci curvature to be a symmetric bilinear form on $W^{q+1}$ whose associated quadratic form $\mathrm{Ric}^{q}(v;W)$ is defined by
\[
    v\mapsto \sum_{i=1}^{q+1}R_{N}(v,v_i,v_i,v)
\]
where $v_1,\ldots,v_{q+1}$ is an orthonormal basis of $W$.
\end{definition}

\section{Convergence Results}

This section aims to generalize convergence and rigidity results in \cite{lee2023rigiditycontractingmapusing} by describing crucial pointwise curvature conditions for target manifolds. This allows us to generalize their result to various curvature or contracting conditions.

We first show that the $k$-nonnegativity of $\phi g-F^*h$ is preserved through the tensor maximum principle(Proposition \ref{tensor maximum principle}). Let's look at the evolution equation of the pull-back metric. If the Ricci tensor $\mathrm{Ric}_{M}$ of $M$ is bounded below by a constant $c>0$, then $\lambda_i^2(\mathrm{Ric}_{M})_{ii}\geq \lambda_i^2c$; if there is a matrix $\tilde{K}_{ij}$ 
such that $\tilde{K}_{ij}\geq (K_{N})_{ij}:=R_{N}(v_i,v_j,v_j,v_i)$, then $-\lambda_i^2\lambda_j^2(K_N)_{ij}\geq -\lambda_i^2\lambda_j^2\tilde{K}_{ij}$; furthermore, if $\sum_{j}\tilde{K}_{ij}\leq c$, these terms can be further simplified to
\[
    \lambda_i^2(\mathrm{Ric}_M)_{ii}-\sum_{j}\lambda_i^2\lambda_j^2(K_N)_{ij}\geq \lambda_i^2\sum_{j}\tilde{K}_{ij}(1-\lambda_j^2).
\]
Therefore, if $\tilde{K}_{ij}$ has sufficiently good properties, we can prove this monotonicity result. Specifically, the properties we need are as follows:

\begin{definition}
Let $0<\phi\leq 1,c> 0$ be fixed constants, $2\leq k\leq \ell$ fixed integers.
\begin{enumerate}[(a)]
    \item 
    For any symmetric matrix $\tilde{K}\in M_{\ell\times \ell}(\mathbb{R})$ with non-negative entries, if the following conditions hold, then we say $\tilde{K}$ satisfies the \textbf{condition $\sigma(k,\ell,c;\phi)$}.
    \begin{itemize}
        \item $\tilde{K}_{11}=\cdots=\tilde{K}_{\ell\ell}=0$;
        \item $\sum_{i=1}^{\ell}\tilde{K}_{ij}\leq c$;
        \item for every vector $u\in (\mathbb{R}_{\geq 0})^k$ with $u_1+\cdots+u_{k}\leq \phi\cdot k$ and every principal $k\times k$-submatrix $A$ of $\tilde{K}$,
        \[
        u^{t}A(\mathds{1}_{k}-u)\geq 0,\ \ \  -(*)
        \]
        where $\mathds{1}_{k}\in M_{k\times 1}(\mathbb{R})$ denotes the column vector with all entries equal to 1.
    \end{itemize}
    \item When $k=\ell$ and the equality in (*) holds if and only if $u=0$ or $u=\phi\cdot \mathds{1}_{\ell}$, then we say $\tilde{K}$ satisfies the \textbf{strong condition $\sigma(\ell,\ell,c;\phi)$}.
    When $k<\ell$, and the equality in (*) holds if and only if $u=0$ or $u_1+\cdots+u_k=\phi\cdot k$, we also say $\tilde{K}$ satisfies the \textbf{strong condition $\sigma(k,\ell,c;\phi)$}.
    \item 
    For a Riemannian manifold $N^n$, we say $N$ satisfies 
 the \textbf{(strong) condition $\sigma(k,\ell,c;\phi)$} if for any set of $\ell$ orthonormal vectors $\{v_1,\ldots,v_{\ell}\}$ within a tangent space $T_{p}N$ of $N$, there exists a matrix $\tilde{K}_{ij}$ satisfying the (strong) condition $\sigma(k,\ell,c;\phi)$ such that $\tilde{K}_{ij}\geq (K_N)_{ij}:=R_{N}(v_i,v_j,v_j,v_i)$.
\end{enumerate}
We denote $\sigma(k,\ell,c;1)$ by $\sigma(k,\ell,c)$ as an abbreviation.
\end{definition}

According to the definition, conditions $\sigma$ yield upper bounds for positive
part of partial Ricci curvature.

\begin{lemma}\label{upper bound of positive part of partial Ricci curvature}
If $(N^n,h)$ is a manifold satisfying the condition $\sigma(k,\ell,c;\phi)$, then $(\mathrm{Ric}_{N}^{\ell-1})_{+}\leq c$.
\begin{proof}
For every orthonormal vectors $v_1,\ldots,v_{\ell}$ in $T_{p}N$, let $\tilde{K}$ be a matrix satisfying the condition $\sigma$ and $\tilde{K}\geq K_N$. Then we have
\[
    (\mathrm{Ric}_{N}^{\ell-1})_{+}(v_1;v_2,\ldots,v_{\ell})=\sum_{i=2}^{\ell}\max(K_{N}(v_1,v_i),0)\leq \sum_{i=2}^{\ell}\tilde{K}_{1i}\leq c.
\]
\end{proof}
\end{lemma}

The following key lemma establishes the preserved condition in our theory.

\begin{lemma}\label{k-nonnegativity is preserved}
Let $(M^m,g)$ and $(N^n,h)$ be closed connected manifolds. Assume $\mathrm{Ric}_M\geq c$ and that $N$ satisfies the condition $\sigma(k,\ell,c;\phi)$ for some $\min(n,m)\leq\ell\leq n$. Then $\lambda_1^2+\cdots+\lambda_k^2\leq \phi k$ is preserved under the harmonic map heat flow.
\begin{proof}
First, consider the SVD frame of $F$ on $M$. Then the $\min(k,m)$-nonnegativity of $\alpha$ is equivalent to $\lambda_1^2+\cdots+\lambda_k^2\leq k\cdot \phi$.

By the tensor maximum principle (Proposition \ref{tensor maximum principle}), it suffices to show that if $\alpha$ is $\min(k,m)$-nonnegative, then $\sum_{i=1}^{\min(k,m)}\left(\left( \frac{\partial}{\partial t}-\Delta \right)\alpha\right)_{ii}\geq 0$. By Corollary \ref{evolution equation written in singular values}, 
\begin{align*}
    \sum_{i=1}^{\min(k,m)}\left(\left( \frac{\partial}{\partial t}-\Delta \right)\alpha\right)_{ii}&\geq 2\sum_{i=1}^{\min(k,m)}\lambda_i^2(\mathrm{Ric}_M)_{ii}-2\sum_{i=1}^{k}\sum_{j=1}^{\ell}\lambda_i^2\tilde{K}_{ij}\lambda_{j}^2\\
    &\geq 2\sum_{i=1}^{k}\sum_{j=1}^{\ell }\lambda_i^2\tilde{K}_{ij}(1-\lambda_{j}^2)
\end{align*}
Since $\lambda_1^2+\cdots+\lambda_{k}^2\leq \phi\cdot k$, $\lambda_{k+1},\ldots,\lambda_{\ell}\leq\phi\leq 1$. Therefore, according to the definition of condition $\sigma$,
\begin{align*}
\sum_{i=1}^{k}\left(\left( \frac{\partial}{\partial t}-\Delta \right)\alpha\right)_{ii}&\geq \sum_{i,j}^{k}\lambda_i^2\tilde{K}_{ij}(1-\lambda_{j}^2)\geq 0,
\end{align*}
and the $\min(k,m)$-nonnegativity of $\alpha:=\phi\cdot g-F^*h$ is preserved under the harmonic map heat flow.
\end{proof}
\end{lemma}

The preservation of $\min(k,m)$-nonnegativity of $\alpha$ yields an upper bound for energy density, and hence the higher order derivatives are bounded by the regularity theory of harmonic map heat flow. 
\begin{theorem}[\cite{MR1428088}, Theorem 2.2]\label{estimate of higher order derivatives}
Let $F_{t}:(M,g)\to (N,h)(t\in [0, T))$ be a maximal extended harmonic map heat flow of a closed manifold into another closed manifold $N$. If $|dF|$ is uniformly bounded, then $T=\infty$ and the norms of the higher order derivatives $|\nabla^{(k-1)}dF|$ are also uniformly bounded.
\end{theorem}
Therefore, the long-time existence and convergence follow are proved.
\begin{proposition}\label{existence of limiting map}
Let $(M^m,g)$ and $(N^n,h)$ be closed connected manifolds. Assume $\mathrm{Ric}_M\geq c$ and $N$ satisfies the condition $\sigma(k,\ell,c;\phi)$  for some $\ell\geq \min(n,m)$. Then for every smooth map $F\colon M\to N$ with $\lambda_1^2+\cdots+\lambda_k^2\leq \phi k$, the harmonic map heat flow $F_{t}$ exists for all time, and there is a sequence $t_{k}\to\infty$ such that $F_{t_k}$ smoothly converges to a harmonic map $F_{\infty}$ with $\lambda_1^2+\cdots+\lambda_k^2\leq \phi k$.
\begin{proof}
By Theorem \ref{estimate of higher order derivatives} and Lemma \ref{k-nonnegativity is preserved}, the higher-order derivatives of $F_{t}$ are uniformly bounded, hence there is a sequence $t_{k}\to\infty$ such that $F_{t_k}$ smoothly converges to a harmonic map $F_{\infty}$ by Arzel\`a-Ascoli Theorem. Moreover, since $F_{t_k}$ converges smoothly to $F_{\infty}$, this inequality $\lambda_1^2+\cdots+\lambda_k^2\leq \phi k$ also applies to $F_\infty$.
\end{proof}
\end{proposition}
We can further characterize this limit by using the equality condition in the inequality.
\begin{lemma}\label{characterizing limiting map}
Let $(M^m,g)$ and $(N^n,h)$ be closed connected manifolds. Assume $\mathrm{Ric}_{M}\geq c$ and that $N$ satisfies the condition $\sigma(k,\ell,c;\phi)$ for some $\ell\geq\min(n,m)$. Then for every harmonic map $F\colon M\to N$ with $\lambda_1^2+\cdots+\lambda_k^2\leq \phi k$ and a SVD frame of $F$ at $p$, we have
\begin{enumerate}[(i)]
    \item $(\nabla dF)_{p}=0$;
    \item $((\mathrm{Ric}_M)_{ii}-c)\lambda_i^2=0$ for all $1\leq i\leq \min(n,m)$;
    \item $\sum_{i,j\in I}\lambda_i^2\tilde{K}_{ij}(1-\lambda_j^2)=0$ for every $I\subset \{1,\ldots,\ell\}$ containing $k$ indices and $\tilde{K}_{ij}$ is the matrix satisfying the condition $\sigma$ such that $\tilde{K}_{ij}\geq K_N$;
    \item $\lambda_i^2(\tilde{K}_{ij}-(K_{N})_{ij})\lambda_{j}^2=0$ for every $1\leq i,j\leq \ell$.
\end{enumerate}
\begin{proof}
Since $F$ is harmonic, $F_{t}=F$ is a harmonic map heat flow. Consider the SVD frame of $F$ on $M$. By Corollary \ref{evolution equation written in singular values}, 
{\small
\begin{align*}
    0&=\int_{M}\left( \frac{\partial}{\partial t}-\Delta \right)e(F_t)\\
    &=\int_{M}\left(-|\nabla dF|^2-\sum_{i=1}^{\min(n,m)}\lambda_i^2(\mathrm{Ric}_{M})_{ii}+\sum_{i,j=1}^{\min(n,m)}\lambda_i^2\lambda_j^2(K_{N})_{ij}\right)\\
    &=\int_{M}\left(-|\nabla dF|^2-\sum_{i=1}^{\min(n,m)}\lambda_i^2((\mathrm{Ric}_M)_{ii}-c)-\sum_{i=1}^{\ell}c\lambda_i^2-\sum_{i,j=1}^{\ell}\lambda_i^2\lambda_j^2(\tilde{K}_{ij}-(K_N)_{ij})+\sum_{i,j=1}^{\ell}\lambda_i^2\lambda_j^2(\tilde{K}_{ij})\right)\\
    &\leq \int_{M}\left(-|\nabla dF|^2-\sum_{i=1}^{\min(n,m)}\lambda_i^2((\mathrm{Ric}_M)_{ii}-c)-\sum_{i,j=1}^{\ell}\lambda_i^2\lambda_j^2(\tilde{K}_{ij}-(K_N)_{ij})-\sum_{i,j=1}^{\ell}\lambda_i^2\tilde{K}_{ij}(1-\lambda_j^2)\right)\\
    &\leq \int_{M}\left(-|\nabla dF|^2-\sum_{i=1}^{\min(n,m)}\lambda_i^2((\mathrm{Ric}_M)_{ii}-c)-\sum_{i,j=1}^{\ell}\lambda_i^2\lambda_j^2(\tilde{K}_{ij}-(K_N)_{ij})-\sum_{i,j=1}^{\ell}\lambda_i^2\lambda_j^2(\tilde{K}_{ij}-(K_N)_{ij})\right)\\
    &-\int_{M}C\sum_{
    I\subset \{1,\ldots,\ell\},
    \#I=k
    }\sum_{i,j\in I}\lambda_i^2\tilde{K}_{ij}(1-\lambda_j^2)\leq 0\\
\end{align*}
}
where $C:={\binom{\ell-2}{k-2}}^{-1}$. Each term in the last two rows is non-positive, so they are all zero by the monotonicity.
\end{proof}
\end{lemma}

Therefore, we can obtain one rigidity result under condition $\sigma$.
\begin{theorem}\label{first convergence result}
Let $(M^m,g)$ and $(N^n,h)$ be closed connected manifolds. Assume $\mathrm{Ric}_M\geq c$ with strict inequality $(\mathrm{Ric}_M)_{p_0}>c$ at some $p_0\in M$ and $N$ satisfies the condition $\sigma(k,\ell,c;\phi)$ for some $\ell\geq\min(n,m)$. Then for every smooth map $F\colon M\to N$ with $\lambda_1^2+\cdots+\lambda_k^2\leq \phi k$, the harmonic map heat flow $F_{t}\colon M\to N(t\in [0,\infty))$ exists for all time and converges to a constant map smoothly.
\begin{proof}
Let $F_{\infty}$ denote the harmonic map in Proposition \ref{existence of limiting map}. By Lemma \ref{characterizing limiting map}, $(dF_{\infty})_{p}=0$. Since $F_{\infty}$ is totally geodesic, it follows that $F_{\infty}$ is a constant map. We thus obtain a homotopy $F_{t}$ between $F=F_0$ and constant map $F_{\infty}$ by Proposition \ref{attraction principle}.
\end{proof}
\end{theorem}

The boundary case can be further characterized under the strong condition $\sigma$.
\begin{lemma}\label{rigidity of harmonic map}
Let $(M^m,g)$ and $(N^n,h)$ be closed connected manifolds. Assume $\mathrm{Ric}_M\geq c$ and $N$ satisfies the strong condition $\sigma(k,\ell,c;\phi)$ for some $\ell\geq\min(n,m)$. Then for every harmonic map $F\colon M\to N$ with $\lambda_1^2+\cdots+\lambda_k^2\leq \phi k$,
\begin{enumerate}
    \item[(i)] either $\phi=1,\ell=\min(n,m)$ and $F$ is a totally geodesic Riemannian subimmersion;
    \item[(ii)] or $F$ is a constant map.
\end{enumerate}
\begin{proof}

If $F$ is constant, then the proof is done, so we assume $F$ is non-constant in the following discussion.

We first show that $F\colon (M,\phi\cdot g)\to (N,h)$ is a totally geodesic Riemannian subimmersion. By Lemma \ref{characterizing limiting map}, $F$ is totally geodesic, so it remains to show that at each point, the first $\min(m,n)$ singular values are either all $\phi$ or 0.

Consider the SVD frame of $F$ on $M$ and $\tilde{K}_{ij}$ used in the condition $\sigma$ of $N$. By Lemma \ref{characterizing limiting map}, for any $I\subset \{1,\ldots,\ell\}$ containing $k$ elements, $\sum_{i,j\in I}\lambda_i^2\tilde{K}_{ij}(1-\lambda_j^2)=0$. We discuss the case $k=\ell$ and $k<\ell$ separately. When $k=\ell$, our claim holds by definition. When $k<\ell$, since $\lambda_{i_1}^2+\cdots+\lambda_{i_k}^2=k\phi$ or $0$, $\lambda_{1}^2-\lambda_{\ell}^2=k\phi$ or $\lambda_{1}^2-\lambda_{\ell}^2=0$. 
If $\lambda_{1}=\lambda_{\ell}$, then it's done. We are going to show the other case is impossible. In this case, since $\lambda_1^2\geq k\phi$, it follows $\lambda_1^2=k\phi$, $\lambda_2^2=\cdots=\lambda_{\ell}^2=0$. However, evaluating the original evolution equation of energy density yields:
\begin{align*}
    0=\int_{M}\Delta e(F_t)=\int_{M}k\phi(\mathrm{Ric}_{M})_{11}\geq \int_{M}k\phi c>0,
\end{align*}
which leads to a contradiction. Thus, $\lambda_{1}=\lambda_{\ell}$ always holds. Since $\lambda_{\ell}\neq 0$, we have $\ell=\min(n,m)$.

By Corollary \ref{evolution equation written in singular values}, we have
\[
    0\leq \int_{M}-c\phi \min(n,m)+c\phi^2\min(n,m)=\int_{M}c\min(n,m)\phi(\phi-1)\leq 0,
\]
and thus $\phi=1$.

\end{proof}
\end{lemma}
The following proposition simplifies the discussion concerning Riemannian submersions between closed manifolds.

\begin{proposition}[R. Hermann, \cite{hermann1960sufficient}, Theorem 1]\label{tg Rimennian submersions are fiber bundles}
Totally geodesic Riemannian submersions between closed manifolds are Riemannian fiber bundles.
\begin{proof}
\end{proof}
\end{proposition}
Thus, we can prove another rigidity result under strong condition $\sigma$.
\begin{theorem}\label{rigidity of contracting maps}
Let $(M^m,g)$ and $(N^n,h)$ be closed connected manifolds. Assume $\mathrm{Ric}_M\geq c$ and $N$ satisfies the strong condition $\sigma(k,\ell,c;\phi)$
for some $\min(n,m)\leq \ell\leq n$. Then for every smooth map $F\colon M\to N$ with $\lambda_1^2+\cdots+\lambda_k^2\leq \phi k$, one of the following holds:
\begin{enumerate}[(i)]
    \item $\phi=1,\ell=n$, and $F\colon (M,g)\to (N,h)$ is a Riemannian fiber bundle, where $\mathrm{Ric}_N=c$ and $K_{N}\geq 0$;
    \item $\phi=1,\ell=m$, and $F\colon (M,g)\to (N,h)$ is a totally geodesic isometric immersion with $\mathrm{Ric}_M=c$;
    \item the harmonic map heat flow $F_{t}\colon M\to N(t\in [0,\infty))$ exists for all time and converges to a constant map smoothly.
\end{enumerate}
\begin{proof}
By Proposition \ref{existence of limiting map} and Lemma \ref{rigidity of harmonic map}, the harmonic map heat flow $F_{t}(t\in [0,\infty))$ exists for all time and subsequently converges to some $F_{\infty}$, which is a Riemannian subimmersion or a constant map. When $F_{\infty}$ is constant, case (iii) holds by Proposition \ref{attraction principle}. Therefore, assume $F_{\infty}:(M,\phi\cdot g)\to (N,h)$ is a Riemannian subimmersion from now on.

We claim that $E(F)\leq E(F_{\infty})$. If so, then $F=F_{\infty}$ by Proposition \ref{equal total energy}.
By direct computation, we have $e(F_{\infty})_{g,h}=\frac{\min(n,m)}{2}$. On the other hand, since $\lambda_1^2+\cdots+\lambda_k^2\leq  k$, $e(F)\leq \frac{\min(n,m)}{2}$, and 
\[
    E(F)=\int_{M}e(F)d\mathrm{vol}_M\leq \int_{M}e(F_{\infty})d\mathrm{vol}_{M}=E(F_{\infty}).
\]

The last step is to establish the rigidity of curvature.
When $m\geq n$, by Lemma \ref{upper bound of positive part of partial Ricci curvature}, $(K_{N})_{ij}\geq \tilde{K}_{ij}\geq 0$ and
\begin{align*}
    c&\geq \sup(\mathrm{Ric}_{N}^{n-1})_{+}\geq \inf(\mathrm{Ric}_{N}^{n-1})_{+}\geq\inf\mathrm{Ric}_{N}^{n-1}\\
    &\geq \inf \mathrm{Ric}_M\tag{since $F$ is  totally geodesic}\\
    &\geq c.
\end{align*}
Therefore, $\mathrm{Ric}_N= c=\mathrm{Ric}^{n-1}_{N}$. By Proposition \ref{tg Rimennian submersions are fiber bundles}, $F$ is a Riemannian fiber bundle; on the other hand, when $m\leq n$
\begin{align*}
    c&\leq \inf\mathrm{Ric}_{M}\leq \sup\mathrm{Ric}_{M}\\
    &\leq \sup\mathrm{Ric}_{N}^{\min(n,m)-1}\tag{since $F$ is  totally geodesic}\\
    &\leq c,
\end{align*}
so $\mathrm{Ric}_{M}=c$.
\end{proof}
\end{theorem}

\section{Contracting Conditions}
In this section, we study which conditions $\sigma$ are satisfied under different curvature conditions, focusing on the curvature pinching conditions.

Let us consider the case of $k=2$ as an example, which has been investigated in \cite{lee2024rigidityareanonincreasingmaps}. Let $\tilde{K}\in M_{\ell\times \ell}(\mathbb{R})$ be a symmetric matrix with non-negative entries satisfying $\tilde{K}_{11}=\cdots=\tilde{K}_{\ell}=0, \sum_{i=1}^{\ell}\tilde{K}_{ij}\leq c$. Then for each of its principal $2\times 2$-submatrices 
\[
A=\begin{pmatrix}
    0 & \tilde{K}_{ij}\\
    \tilde{K}_{ij} & 0
\end{pmatrix}
\]
and any vector $u\in \mathbb{R}^2$ with non-negative entries and $u_1+u_2\leq 2$, we have
\begin{align*}
    &\begin{pmatrix}
    u_1 & u_2
    \end{pmatrix}
    \begin{pmatrix}
        0 & K_{ij}\\
        K_{ij} & 0
    \end{pmatrix}
    \begin{pmatrix}
    1-u_1\\
    1-u_2
    \end{pmatrix}\\
    &=K_{ij}(u_1(1-u_2)+u_2(1-u_1))\\
    &=K_{ij}\left(
    \frac{(u_1+u_2)}{2}(2-u_1-u_2)+\frac{(u_1-u_2)^2}{2}
    \right)\geq 0.
\end{align*}
Furthermore, when $K_{ij}>0$, the equality holds if and only if $(u_1,u_2)=(0,0)$ or $(1,1)$. Therefore, for every Riemannian manifold $N$ of positive sectional curvature, if its partial Ricci curvature 
$\mathrm{Ric}_N^{\ell-1}$ is bounded above by a constant $c>0$, then $N$ satisfies the strong condition $\sigma(2,\ell,c)$.

It can be seen that $k=2$ is quite a special case. When we try to extend the above method to the case of $k>2$, the first problem is that the submatrix is no longer always a multiple of $1-\delta_{ij}$. We address this problem by considering curvature pinching conditions. 

First, by considering each direction of vector u, let us investigate for any given matrix $A$, $A$ satisfies which conditions $\sigma$.

\begin{definition}
Let $\tilde{K}\in M_{\ell\times \ell}(\mathbb{R})$ be a symmetric matrix with non-negative entries such that $\tilde{K}_{11}=\cdots=\tilde{K}_{\ell\ell}=0$.

\begin{enumerate}[(a)]
\item
For any principal $k\times k$-submatrix $A$ of $\tilde{K}$ and any vector $v$ in the standard simplex $\mathcal{S}^{k-1}:=\{v\in \mathbb{R}^{k}\big|v_{i}\geq 0,\sum_{i=1}^{k}v_{i}=k\}$, we define
\[
    \phi_{A}(v):=\sup\{0\leq \phi\leq \infty\big| (\phi\cdot v^{t})A\mathds{1}_{k}-(\phi\cdot v^t)A(\phi\cdot v)\geq 0,
    \}
\]
\item and we define
\[
    \phi(A):=\sup_{v\in \mathcal{S}^{k-1}}
    \phi_A(v).
\]
\item 
Now, we define $\phi_k(\tilde{K}):=\min_{A}\phi(A)$, where $A$ runs over all principal $k\times k$-submatrix of $\tilde{K}$.
\end{enumerate}
\end{definition}

According to our definitions, we can reformulate conditions $\sigma(k,\ell,c;\phi)$ through $\phi_k(\tilde{K})$.

\begin{proposition}\label{condition sigma: expression in terms of phi}
Let $\tilde{K}\in M_{\ell\times \ell}(\mathbb{R})$ be a symmetric matrix with non-negative entries such that $\tilde{K}_{11}=\cdots=\tilde{K}_{\ell\ell}=0,\sum_{j=1}^{\ell}\tilde{K}_{ij}\leq c$ for all $i$. Then $\tilde{K}$ satisfies the condition $\sigma(k,\ell,c;\phi)$ if and only if $\phi\leq \phi_k(\tilde{K})$.
\begin{proof}
Notice that $\tilde{K}$ satisfies the condition $\sigma(k,\ell,c;\phi)$ if and only if for every $v\in \mathcal{S}^{k-1}$, $0\leq \psi\leq \phi$ and principal $k\times k$-submatrix $A$ of $\tilde{K}$,
\[
    (\psi\cdot v^{t})A\mathds{1}_{k}-(\psi\cdot v^t)A(\psi\cdot v)\geq 0,
\]
which is equivalent to say $\phi(A)\leq \phi$ for every principal $k\times k$-submatrix $A$ of $\tilde{K}$. Thus, these two conditions are equivalent.
\end{proof}
\end{proposition}

Since $(\phi\cdot v^{t})A\mathds{1}_{k}-(\phi\cdot v^t)A(\phi\cdot v)$ is a quadric polynomial of $\phi$ of non-positive leading coefficient, we have:

\begin{lemma}\label{condition sigma: solve phi_A(v)}
Let $A\in M_{k\times k}(\mathbb{R})$ be a symmetric matrix with non-negative entries such that $A_{11}=\cdots=A_{kk}=0$, $k\geq 2$.
Then for every $v\in \mathcal{S}^{k-1}$, we have
\begin{enumerate}[(i)]
    \item $\phi_{A}(v)=\frac{v^{t}A\mathds{1}_{k}}{v^{t}Av}$. (Here, $\frac{v^{t}A\mathds{1}_{k}}{v^{t}Av}$ is defined to be $\infty$ if $v^{t}Av=0$)
    \item Moreover, if $A_{ij}>0$ for $i\neq j$, then 
    \begin{align*}
        (\phi v^{t})A\mathds{1}_{k}-(\phi v)^{t}A(\phi v)&>0\text{ for all $0<\phi<\phi_{A}(v)$;}\\
        (\phi v^{t})A\mathds{1}_{k}-(\phi v)^{t}A(\phi v)&<0\text{ for all $\phi_{A}(v)<\phi<\infty$.}
    \end{align*}
\end{enumerate}
\begin{proof}
\begin{enumerate}[(i)]
\item 
When $v^{t}Av=0$, we get $\phi_{A}(v)=\infty$ directly. When $v^{t}Av>0$, 
\[
    (\phi v^t)A\mathds{1}_k-(\phi\cdot v^t)A(\phi\cdot v)=\phi(v^tA\mathds{1}_k)-\phi^2(v^tAv)=(v^tAv)\phi (\frac{v^tA\mathds{1}_k}{v^tAv}-\phi),
\]
so (i) holds.
\item 
In this case, every entry in $A\mathds{1}_{k}$ is positive, and hence $v^{t}A\mathds{1}_k>0$. When $v^{t}Av>0$, using the equality in the proof of (i), we deduce (ii); when $v^{t}Av=0$,
\[
    (\phi v^t)A\mathds{1}_k-(\phi\cdot v^t)A(\phi\cdot v)=\phi(v^tA\mathds{1}_k)>0
\]
for all $\phi>0$, so (ii) still holds.
\end{enumerate}
\end{proof}
\end{lemma}

For any manifold $N$ of positive sectional curvature, if $\tilde{K}$ is a matrix such that $\tilde{K}_{ij}\geq (K_N)_{ij}$, and a principal $k\times k$-submatrix $A$ of $\tilde{K}$, $A$ always satisfies the extra condition in Lemma \ref{condition sigma: solve phi_A(v)}-(ii) always holds on $A$. Therefore, we can reformulate strong conditions $\sigma$ through $\phi_k(\tilde{K})$ and $\phi_{A}(v)$ in this case.

\begin{proposition}\label{strong condition sigma: expression in terms of phi}
Let $\tilde{K}\in M_{\ell\times \ell}(\mathbb{R})$ be a symmetric matrix such that $\tilde{K}_{11}=\cdots=\tilde{K}_{\ell\ell}=0,\sum_{j=1}^{\ell}\tilde{K}_{ij}\leq c$, $\tilde{K}_{ij}> 0$ for all $i\neq j$ for all $i$, $2\leq k\leq \ell$.  Then the following statements hold:
\begin{enumerate}[(i)]
    \item when $k<\ell$, $\tilde{K}$ satisfies the strong condition $\sigma(k,\ell,c;\phi)$ if and only if $\tilde{K}$ satisfies the condition $\sigma(k,\ell,c;\phi)$.
    \item when $k=\ell$, $\tilde{K}$ satisfies the strong condition $\sigma(\ell,\ell,c;\phi)$ if and only if $\phi_{\tilde{K}}(v)\leq \phi$ for every $v\in \mathcal{S}^{\ell-1}$ and $\phi$ only can be attained by $\mathds{1}_{\ell}$.
\end{enumerate}
\begin{proof}
\begin{enumerate}[(i)]
    \item By Proposition \ref{condition sigma: expression in terms of phi}, $\tilde{K}$ satisfies the condition $\sigma(k,\ell,c;\phi)$ if and only if $\phi\leq \phi_{k}(\tilde{K})$. By Lemma \ref{condition sigma: solve phi_A(v)}, it is equivalent to 
    \[
    (v^{t})A\mathds{1}_{k}-(v)^{t}A(v)>0
    \]
    for every principal $k\times k$-submatrix $A$ of $\tilde{K}$ and $0<\sum_{i=1}v_i<\phi$. That is, the definition of strong condition $\sigma(k,\ell,c;\phi)$.
    \item The proof is essentially similar to (i). In this case, since $\phi$ only can be attained by $\mathds{1}_{\ell}$, for every $v\in \mathcal{S}^{\ell-1}\setminus \{\mathds{1}_{\ell}\}$, $\phi_{\tilde{K}}(v)>\phi$, and we have
    \[
        (v^{t})A\mathds{1}_{k}-(v)^{t}A(v)>0.
    \]
    Therefore, these two statements are equivalent.
\end{enumerate}
\end{proof}
\end{proposition}
By the formula of $\phi_A$, we also can show the minimizer exists when $A_{ij}>0$ for $i\neq j$, which allows us to apply the variational method.

\begin{proposition}\label{minimizer of phi_A}
Let $A\in M_{k\times k}(\mathbb{R})$ be a symmetric matrix such that $A_{11}=\cdots=A_{kk}=0$, $A_{ij}> 0$ for all $i\neq j$ for all $i$, where $2\leq k$. Then
\[
\phi(A)=\inf_{v\in \mathcal{S}_{k}}\phi_{A}(v)=\min_{v\in \mathcal{S}_{k}}\phi_{A}(v)>0
\]
\begin{proof}
As in the proof of Lemma \ref{condition sigma: solve phi_A(v)},
since 
$\lim_{v\to v_0}\phi_{A}(v)=\infty$ for every $v_0\in \mathcal{S}^{k-1}$ with $v_0^tAv_0=0$, $\phi_{A}:\mathcal{S}_k\to [0,\infty]$ is a continuous function on compact set, so the minimizer exists.
\end{proof}
\end{proposition}

Based on the above discussion, we build on a translation between the conditions $\sigma$ and minimizing problems for $\phi_A$. We are going to apply this technique to investigate conditions $\sigma$ in the remaining of this section.

According to the curvature conditions of manifolds, let us make some terminologies for matrices. Corresponding to Einstein manifolds, for a symmetric matrix $A\in M_{k\times k}(\mathbb{R})$ with non-negative entries such that $A_{11}=\cdots=A_{kk}=0$, if $\sum_{j=1}^{k}A_{1j}=\cdots=\sum_{j=1}^{k}A_{kj}=c$, we say $A$ is \textbf{Einstein (of constant $c$)} .
Similarly, for positive curvature pinching condition, for a symmetric matrix $\tilde{K}\in M_{\ell\times\ell}(\mathbb{R})$ and a constant $r\geq 1$, if $\tilde{K}_{11}=\cdots=\tilde{K}_{\ell\ell}=0$, $\tilde{K}_{ij}>0$ for all $i\neq j$, and $\tilde{K}_{ij}\leq r\tilde{K}_{i'j'}$ for any pairs $(i,j),(i',j')$ of distinct elements, we say $\tilde{K}$ is \textbf{of pinching $r$}. 

Notice that $A$ is Einstein if and only if $\mathds{1}_k$ is an eigenvector of $A$. Hence, we can prove every matrix satisfying conditions $\sigma(k,k,c)$ is Einstein:
\begin{lemma}\label{condition sigma(k,l,c) are Einstein}
Let $A\in M_{k\times k}(\mathbb{R})$ be a symmetric matrix with non-negative entries such that $A_{11}=\cdots=A_{kk}=0$. 
\begin{enumerate}[(i)]
    \item If $v\in \mathcal{S}^{k-1}$ is a minimizer of $\phi_A$, then $A(\mathds{1}_k-2\phi_A(v)v)$ lies in the normal cone $N_{v}\mathcal{S}^{k-1}$ of $\mathcal{S}^{k-1}$.
    \item Therefore, if $\mathds{1}_{k}$ attains minimum of $\phi_{A}$, then $A$ is Einstein.
\end{enumerate}
\begin{proof}
When $A=0$, the above statements hold trivially, so we assume $A\neq 0$. In this case, since $\mathds{1}_{k}^{t}A\mathds{1}_{k}>0$, $\min \phi_A<\infty$, and $v^tAv>0$. So, we can differentiate $\phi_{A}$. By definition of normal cone, we have
\[
    \mathrm{grad}\phi_A=\frac{A\mathds{1}_k}{v^tAv}-\frac{v^tA\mathds{1}_k}{(v^tAv)}\cdot \frac{2Av}{v^tAv}=\frac{1}{v^tAv}A(\mathds{1}_k-2\phi_A(v)v)\in N_{v}\mathcal{S}^{k-1}.
\]
By (i), we have $A(-\mathds{1}_k)\in N_{\mathds{1}_k}\mathcal{S}^{k-1}=\langle \mathds{1}_k \rangle$, i.e. $A$ is Einstein.
\end{proof}
\end{lemma}

Therefore, when researching condition $\sigma(k,k,c)$, we may assume $A$ is Einstein.

\begin{corollary}\label{phi_A=1}
Let $A\in M_{k\times k}(\mathbb{R})$ be a symmetric matrix with non-negative entries such that $A_{11}=\cdots=A_{kk}=0$. Then
\begin{enumerate}[(i)]
    \item $\phi(A)\geq 1$ if and only if $A$ is Einstein and the second eigenvalue of $A$ is non-positive.
    \item $\phi(A)\geq 1$ and this minimum is only attained by $\mathds{1}_{k}$ if and only if $A$ is Einstein and the second eigenvalue of $A$ is negative.
\end{enumerate}
\begin{proof}
By Lemma \ref{condition sigma(k,l,c) are Einstein}, we may assume $A$ is Einstein.
\begin{enumerate}[(i)]
    \item $\phi(A)\geq 1$ if and only if 
    \[
\phi_A(\mathds{1}_k+w)=\frac{\mathds{1}_k^tA\mathds{1}_k}{\mathds{1}_k^tA\mathds{1}_k+w^tAw}\geq 1
    \]
    for every $\mathds{1}_k+w\in \mathcal{S}^{k-1}$ if and only if $w^{t}Aw\leq 0$ for every $w\in T_{\mathds{1}_k}\mathcal{S}^{k-1}=\langle \mathds{1}_k \rangle^{\perp}$ if and only if the second eigenvalue of $A$ is non-positive.
    \item 
    Similar to (i), $\phi(A)\geq 1$ and this minimum is only attained by $\mathds{1}_{k}$ if and only if 
    \[
\phi_A(\mathds{1}_k+w)=\frac{\mathds{1}_k^tA\mathds{1}_k}{\mathds{1}_k^tA\mathds{1}_k+w^tAw}> 1
    \]
    for every $\mathds{1}_k+w\in \mathcal{S}^{k-1}\setminus \{\mathds{1}_k\}$ if and only if the second eigenvalue of $A$ is negative.
\end{enumerate}
\end{proof}
\end{corollary}

We need the following lemma to apply Corollary \ref{phi_A=1} to general conditions $\sigma(k,\ell,c)$.

\begin{lemma}\label{partial Einstein implies space form}
Let $\tilde{K}\in M_{\ell\times \ell}(\mathbb{R})$ be a symmetric matrix with non-negative entries such that $\tilde{K}_{11}=\cdots=\tilde{K}_{\ell\ell}=0,\sum_{j=1}^{\ell}\tilde{K}_{ij}\leq c$ for all $i$.
If there is an integer $2<k<\ell$ such that every principal $k\times k$-submatrix of $\tilde{K}$ is Einstein, then $\tilde{K}$ is either 0 or of pinching $1$.
\begin{proof}
For every two pairs $(i,j)$ and $(p,q)$ of distinct elements, we want to show $\tilde{K}_{ij}=\tilde{K}_{pq}$. Since $\tilde{K}$ is symmetric, when $\{i,j\}=\{p,q\}$, we have $\tilde{K}_{ij}=\tilde{K}_{pq}$. If $\{i,j\}\cap \{p,q\}=\phi$, then by considering $\tilde{K}_{ij}, \tilde{K}_{iq},\tilde{K}_{pq}$, it suffices to show $\tilde{K}_{ji}=\tilde{K}_{qi}$ for any distinct three elements $i,j,q$.

For any four distinct elements $i,j,p,q$, extend them to distinct $k+1$ elements $\{ a_1=i,a_2=j,\ldots,a_{k}=p,a_{k+1}=q\}$. By assumption, we have
\begin{align*}
    \tilde{K}_{a_2a_1}+\cdots+\tilde{K}_{a_{k}a_1}&=\tilde{K}_{a_1a_2}+\tilde{K}_{a_3a_2}+\cdots+\tilde{K}_{a_ka_2}\\
    \tilde{K}_{a_2a_1}+\cdots+\tilde{K}_{a_{k-1}a_1}+\tilde{K}_{a_{k+1}a_1}&=\tilde{K}_{a_1a_2}+\tilde{K}_{a_3a_2}+\cdots+\tilde{K}_{a_{k-1}a_2}+\tilde{K}_{a_{k+1}a_{2}},
\end{align*}
and thus
\[
    \tilde{K}_{qi}-\tilde{K}_{pi}=\tilde{K}_{qj}-\tilde{K}_{pj}.
\]

Now, for any distinct three elements $i,j,p$, extend them to distinct $k$ elements $\{a_1=i,a_2=j,a_3=p,a_4,\ldots,a_{k}\}$. Let $\alpha:=\tilde{K}_{pj}-\tilde{K}_{p_i}$. Since every principal $k\times k$-submatrix is Einstein, we have
\[
    \tilde{K}_{a_2a_1}+\cdots+\tilde{K}_{a_ka_1}=\tilde{K}_{a_1a_2}+\tilde{K}_{a_3a_2}+\cdots+\tilde{K}_{a_ka_2}=\tilde{K}_{a_2a_1}+\cdots+\tilde{K}_{a_ka_1}+(k-2)\alpha.
\]
Hence, $(k-2)\alpha=0$, and $\alpha=\tilde{K}_{pj}-\tilde{K}_{p_i}=0$.
\end{proof}
\end{lemma}

Now, we can describe matrices satisfying condition $\sigma(k,\ell,c)$ completely.

\begin{proposition}\label{condition sigma, sigma(k,l,c) for matrices}
Matrices satisfy conditions $\sigma(k,\ell,c)$ are as follows:
\begin{enumerate}[(i)]
    \item Matrices satisfying the condition $\sigma(2,\ell,c)$ are $\tilde{K}\in M_{\ell\times\ell}(\mathbb{R})$ such that $\tilde{K}=\tilde{K}^t, \tilde{K}_{ii}=0, \tilde{K}_{ij}\geq 0$ and $\sum_{j=1}^{\ell}\tilde{K}_{ij}\leq c$. 
    \item Matrices satisfying the strong condition $\sigma(2,\ell,c)$ are $\tilde{K}\in M_{\ell\times\ell}(\mathbb{R})$ such that $\tilde{K}=\tilde{K}^t, \tilde{K}_{ii}=0, \tilde{K}_{ij}>0$ for all $i\neq j$ and $\sum_{j=1}^{\ell}\tilde{K}_{ij}\leq c$. 
    \item For $2<k<\ell$, matrices satisfying the condition $\sigma(k,\ell,\kappa(\ell-1))$ are $\tilde{K}_{ij}=\kappa'(1-\delta_{ij})$, where $0\leq \kappa'\leq \kappa$.
    \item Matrices satisfying the (strong) condition $\sigma(\ell,\ell,c)$ are Einstein matrices of constants at most $c$ and (negative) non-positive second eigenvalue. In particular,
    Einstein matrices of constants at most $c$ and pinching (less than) at most $2$ satisfy the (strong) condition $\sigma(n,n,c)$.
\end{enumerate}
\begin{proof}
Let $\tilde{K}\in M_{\ell\times \ell}(\mathbb{R})$ be a symmetric matrix with non-negative entries such that $\tilde{K}_{11}=\cdots=\tilde{K}_{\ell\ell}=0,\sum_{j=1}^{\ell}\tilde{K}_{ij}\leq c$ for all $i$. We discuss each condition $\sigma$ respectively.
\begin{enumerate}[(i)]
    \item 
     For every $2\times 2$ principal submatrix $A$ of $\tilde{K}$, $A=\begin{pmatrix}
        0 & a \\
        a & 0
    \end{pmatrix}$ has the second eigenvalue $-a\leq 0$, so $\phi_2(\tilde{K})\geq 1$ by Corollary \ref{phi_A=1}, and (i) is proved.
    \item 
    As in (i), $\tilde{K}$ satisfies the strong condition $\sigma(2,\ell,c)$ if and only if for every principal $2\times 2$-submatrix $A$ and $\phi_A$, $1$ is only attained by $\mathds{1}_2$ if and only if $\tilde{K}_{ij}>0$ for every $i\neq j$.
    \item 
    When $\tilde{K}_{ij}=\kappa'(1-\delta_{ij})$, its $k\times k$ principal submatrices $A$ are also of the form $\kappa'(1-\delta_{ij})$, whose second eigenvalue is $-\kappa'$, and $\phi(A)\geq 1$ by Corollary \ref{phi_A=1}. Hence, $\tilde{K}_{ij}$ satisfies the condition $\sigma(k,\ell,\kappa(\ell-1))$.
    On the other hand, if $\tilde{K}$ satisfies condition $\sigma(k,\ell,\kappa(\ell-1))$, then its principal $k\times k$-submatrix are Einstein by Lemma \ref{condition sigma(k,l,c) are Einstein}, and thus $\tilde{K}_{ij}=\kappa'(1-\delta_{ij})$ for some $0\leq \kappa'\leq \kappa$ by Lemma \ref{partial Einstein implies space form}.
    \item 
    The first statement follows from Corollary \ref{phi_A=1} directly. For the second statement, we may assume $\min_{i\neq j}\tilde{K}_{ij}=1$ and decompose $\tilde{K}$ into
    \[
        \tilde{K}_{ij}=(1-\delta_{ij})+(\tilde{K}_{ij}-(1-\delta_{ij})),
    \]
    where the latter term is a multiple of transition matrix and hence of second eigenvalue at most $r-1$. Thus, the second eigenvalue of $\tilde{K}$ is at most $r-2$, and the second statement is proved by the first one.
\end{enumerate}
\end{proof}
\end{proposition}


Apply Proposition \ref{condition sigma, sigma(k,l,c) for matrices} to manifolds. Then we get the following corollary.

\begin{corollary}\label{condition sigma for manifolds}
Let $(N^n,h)$ be a Riemannian manifold with $\sup(\mathrm{Ric}_{N}^{n-1})_{+}<\infty$. Then the following statements hold:
\begin{enumerate}[(i)]
    \item $N$ always satisfies the condition $\sigma(2,\ell,\sup(\mathrm{Ric}_{N}^{\ell-1})_{+})$.
    \item for every $c>\sup(\mathrm{Ric}_{N}^{\ell-1})_{+}$, $N$ always satisfies the strong condition $\sigma(2,\ell,c)$, and $N$ satisfies the strong condition $\sigma(2,\ell,\sup(\mathrm{Ric}_{N}^{\ell-1})_{+})$ if and only if the maximum of $(\mathrm{Ric}_{N}^{\ell-1})_{+}$ is only attained by vectors $v_1,\ldots,v_{\ell}$ such that $K_N(v_1,v_2),\ldots,K_N(v_1,v_{\ell})>0$
    \item for $2<k<\ell\leq n$, $N$ satisfies the condition $\sigma(k,\ell,(\ell-1)\kappa)$ if and only if $K_N\leq \kappa$. In this case, when $\kappa>0$, $N$ satisfies all strong conditions $\sigma(k,\ell,(\ell-1)\kappa)$.
    \item if $N$ is an Einstein manifold of positive sectional curvature with $\sup K_N/\inf K_N<2$, then $N$ satisfies the strong condition $\sigma(n,n,c)$.
    \item[(iv)'] if $N$ is an Einstein manifold of positive sectional curvature with $\sup K_N/\inf K_N\leq 2$, then $N$ satisfies the condition $\sigma(n,n,c)$.
\end{enumerate}

\begin{proof}
\begin{enumerate}[(i)]
\item Denote $\sup(\mathrm{Ric}_{N}^{\ell-1})_{+}$ by $c$.
For every orthonormal vectors $v_1,\ldots,v_{\ell}\in T_pN$, 
\[
\tilde{K}_{ij}:=\max(K_N(v_i,v_j),0)
\]
is a matrix satisfying the condition $\sigma(2,\ell,c)$ and $\tilde{K}_{ij}\geq K_N(v_i,v_j)$. Hence, $N$ satisfies the condition $\sigma(2,\ell,c)$.
\item 
For the first statement, let $\epsilon>0$ such that $\sup(\mathrm{Ric}_{N}^{\ell-1})_{+}+(\ell-1)\epsilon<c$. Then $\tilde{K}_{ij}:=\max(K_N(v_i,v_j),0)+\epsilon(1-\delta_{ij})$ is our goal.

Assume $N$ satisfies the strong condition $\sigma(2,\ell,c)$ and $v_1,\ldots,v_{\ell}$ attain maximum of $(\mathrm{Ric}_{N}^{\ell-1})_{+}$. Let $\tilde{K}_{ij}$ be the matrix satisfying strong condition $\sigma(2,\ell,c)$ and $\tilde{K}_{ij}\geq K_N(v_i,v_j)$. Since
\[
    \sup(\mathrm{Ric}_{N}^{\ell-1})_{+}=\sum_{i=2}^{\ell}\max(K_N(v_1,v_i),0)\leq \sum_{i=2}^{\ell}\tilde{K}_{1i}\leq \sup(\mathrm{Ric}_{N}^{\ell-1})_{+},
\]
we have $\tilde{K}_{ij}=\max(K_N(v_1,v_i),0)$. Furthermore, since $\tilde{K}$ satisfies the strong condition $\sigma(2,\ell,c)$, $\tilde{K}_{1i}>0$ for $i=2,\ldots,\ell$, and hence $K_{N}(v_1,v_i)=\tilde{K}_{1i}>0$ for all $2\leq i\leq \ell$.

Conversely, assume the maximum of $(\mathrm{Ric}_{N}^{\ell-1})_{+}$ is only attained by such vectors. Let $v_{1},\ldots,v_{\ell}\in T_{pN}$ be orthonormal vectors. By rearrangement, we may assume 
\begin{align*}
    (\mathrm{Ric}_{N}^{\ell-1})_{+}(v_1;v_2,\ldots,v_\ell)=\cdots=(\mathrm{Ric}_{N}^{\ell-1})_{+}(v_p;v_1,\ldots,\check{v_i},\ldots,v_\ell)&=c\\
    (\mathrm{Ric}_{N}^{\ell-1})_{+}(v_{p+1};v_1,\ldots,\check{v_{i+1}},\ldots,v_\ell),\ldots,(\mathrm{Ric}_{N}^{\ell-1})_{+}(v_\ell;v_1,\ldots,v_{\ell-1})&<c,
\end{align*}
where $0\leq p\leq \ell$. Then, by assumption, $K_{N}(v_j,v_k)>0$ if $j\neq k$ and one of $j,k$ is smaller than $p$. Hence, in order to find suitable $\tilde{K}_{ij}$, we only have to adjust the lower right $(\ell-p)\times(\ell-p)$-principal submatrix.

Let $\epsilon>0$ be a small constant such that 
\[
(\mathrm{Ric}_{N}^{\ell-1})_{+}(v_{p+1};v_1,\ldots,\check{v_{i+1}},\ldots,v_\ell)+\ell\epsilon,\ldots,(\mathrm{Ric}_{N}^{\ell-1})_{+}(v_\ell;v_1,\ldots,v_{\ell-1})+\ell\epsilon\leq c.
\]
Then 
\[
\tilde{K}_{ij}:=\begin{pmatrix}
\max(K_N(v_i,v_j),0)
\end{pmatrix}+
\begin{pmatrix}
0_{p\times p} & 0_{p\times (\ell-p)}\\
0_{(\ell-p)\times p} & (1-\delta_{ij})\epsilon
\end{pmatrix}
\]
is a matrix satisfying the strong condition $\sigma(2,\ell,c)$ and $\tilde{K}_{ij}\geq K_N(v_i,v_j)$.
\item 
Assume $N$ satisfies the condition $\sigma(k,\ell,(\ell-1)\kappa)$. Then for every orthonormal vectors $v_1,\ldots,v_{\ell}$, $K_N(v_i,v_j)\leq \kappa(1-\delta_{ij})$, so $K_{N}\leq \kappa$. Conversely, if $K_{N}\leq \kappa$, $\tilde{K}_{ij}=\kappa(1-\delta_{ij})$ is our goal.
\item[(iv)\&(iv)'] 
$\tilde{K}_{ij}:=K_N(v_i,v_j)$ is our goal by Proposition \ref{condition sigma, sigma(k,l,c) for matrices}.
\end{enumerate}
\end{proof}
\end{corollary}

Corollary \ref{condition sigma for manifolds} demonstrates the constraints of condition $\sigma(k,\ell,c)$. That is why we should consider parameter $\phi$ for weaker curvature conditions.

The main difficulty in investigating general condition $\sigma(k,\ell,c;\phi)$ is that the restrictions to principal submatrices don't preserve algebraic properties of matrices. There are two situations where we can avoid it. The first one is when we consider condition $\sigma(n,n,c;\phi)$ for Einstein manifolds, where we don't restrict the matrices. The second one is solely considering the curvature pinching condition, which is evidently preserved under restriction.


\begin{proposition}\label{condition sigma for Einstein manifolds}
Let $(N^n,h)$ be an Einstein manifold with $\mathrm{Ric}_N=c\cdot h$ and positive sectional curvature $K_N>0$ of pinching $r>2$. Then $N$ satisfies the strong condition $\sigma(n,n,c;\frac{n-2+r}{nr-n})$.
\begin{proof}
For every orthonormal basis $v_1,\ldots,v_n$ of $T_pN$, we are going to show $\tilde{K}_{ij}:=K_{N}(v_i,v_j)$ satisfies the strong condition $\sigma(n,n,c;\frac{n-2+r}{nr-n})$. Notice that it suffices to show $\phi(\tilde{K})>\frac{n-2+r}{nr-n}$.

First, we estimate the second eigenvalue of $\tilde{K}_{ij}$. Let $A:=\max\tilde{K}_{ij}=\tilde{K}_{i_0j_0}, B:=\inf_{i\neq j}\tilde{K}_{ij}$. Then we have
\begin{align*}
    A+(n-2)B&\leq \sum_{j=1,j\neq i_0}^{n}\tilde{K}_{i_0j}\leq c, \text{ and }\\
    A&\leq r\cdot B
\end{align*}
by curvature pinching condition. Decompose $\tilde{K}_{ij}$ into 
\[
B(1-\delta_{ij})+(\tilde{K}_{ij}-B(1-\delta_{ij})).
\]
Then $(\tilde{K}_{ij}-B(1-\delta_{ij}))$ is an Einstein matrix of entries at most $A-B$, and its second eigenvalue (that is, the largest eigenvalue on $\langle \mathds{1}_{n}\rangle^{\perp}$) is at most $A-B$. Hence, the second eigenvalue of $\tilde{K}_{ij}$ is at most $A-2B$. Now, we solve maximum of $A-2B$ subject to $A+(n-2)B\leq c,A\leq r\cdot B$. When $A+(n-2)B$ is fixed, the maximum of $A-2B$ occurs when $A=r\cdot B$, so we assume $A=r\cdot B$. In this case, $(n-2+r)\cdot B\leq c$, and hence the second eigenvalue is smaller than $A-2B=(r-2)\cdot B\leq \frac{r-2}{n-2+r}\cdot c$.

Since $\phi_{\tilde{K}}(v)$ tends to $\infty$ as $v$ tends to corners $e_1,\ldots,e_{k}$ of $\mathcal{S}^{n-1}$, we have
\begin{align*}
    \phi(\tilde{K})&=\inf_{v\in \mathcal{S}^{n-1}\setminus \{e_1,\ldots,e_n\}}\frac{v^{t}\tilde{K}\mathds{1}_{n}}{v^{t}\tilde{K}v}\\
    &=\inf_{w\in (\mathcal{S}^{n-1}\setminus \{e_1,\ldots,e_n\})-\mathds{1}_k}\frac{cn}{cn+w^tAw}\\
    &>\frac{cn}{cn+(n^2-n)\left(\frac{r-2}{n-2+r}\right)c}\\
    &=\frac{n-2+r}{nr-n}.
\end{align*}
Therefore, $\tilde{K}$ satisfies the strong condition $\sigma(n,n,c;\frac{n-2+r}{nr-n})$, and so is $N$.
\end{proof}
\end{proposition}

The parameter $\phi$ for general curvature pinching conditions can be characterized by some extreme cases. Given a positive integer $k$, we let $\mathcal{I}_{k}:=\{(i,j)\big|1\leq i<j\leq k\}$ and for every $J\subset \mathcal{I}_{k}$, define $A_{J}(r)\in M_{k\times k}(\mathbb{R})$ by
\begin{align*}
    A_{J}(r)_{ii}&:=0\\
    A_{J}(r)_{ij}&=A_{J}(r)_{ji}:=r\tag{ for every $(i,j)\in J$}\\
    A_{J}(r)_{ij}&:=1\tag{for other $(i,j)$}.
\end{align*}
Then we have the following proposition:

\begin{proposition}\label{condition sigma: comparison to A_J(r)}
For any $k\times k$-matrix $A$ of pinching $r$,
\[
    \phi(A)\geq \min_{J\subset \mathcal{I}_k}\phi(A_{J}(r))=:\phi_k(r).
\]
Moreover, for every $k\geq 3$, $\phi_k(r)>2/k$.
\begin{proof}
For convenience, we may assume $\inf_{i\neq j}A_{ij}=1$. Then $A_{ij}\leq r$. For every vector $v\in (\mathbb{R}_{\geq 0})^k$ with $\sum_{i=1}^{k}v_i\leq \phi_k(r)$, we have
\begin{align*}
    v^tA(\mathds{1}_k-v)&=\sum_{(i,j)\in I}A_{ij}(v_i+v_j-2v_iv_j)\\
    &\geq \sum_{(i,j)\in J}r(v_i+v_j-2v_iv_j)+\sum_{(i,j)\in \mathcal{I}_k\setminus J}(v_i+v_j-2v_iv_j)\tag{where $J:=\{(i,j)\in I\big|v_i+v_j-2v_iv_j<0\}$}\\
    &\geq 0\tag{since $\sum_{i=1}^{k}v_i\leq \phi(A_J(r))$},
\end{align*}
so $\phi_{k}(A)\geq \phi_k(r).$

Now, suppose $\phi_k(r)\leq 2/k$. Then, the minimum of some $\phi_{A_J(r)}$ is attained by a vector $v\in (\mathbb{R}_{\geq 0})^k$. In this case, $v^tA_{J}(r)(\mathds{1}_k-v)=0$. Since each $v_i+v_j-2v_iv_j$ is non-negative, they are all zero, and $v=0$. However, this is impossible since one can show $\mathds{0}_k$ is an isolated zero of $v^tA_{J}(r)(\mathds{1}_k-v)$ by its first variation. Hence, $\phi_k(r)>2/k$.
\end{proof}
\end{proposition}

\begin{corollary}
    For every Riemannian manifold $N^n$ of positive sectional curvature, curvature pinching $\kappa\geq 1$, and $\mathrm{Ric}_N^{\ell-1}\leq c$, $N$ satisfies the condition $\sigma(k,\ell,c;\phi_k(\kappa))$. 
\begin{proof}
    This corollary is directly proved by Proposition \ref{condition sigma: comparison to A_J(r)}. 
\end{proof}
\end{corollary}

\begin{corollary}\label{condition sigma for 1/4-pinching}
Let $N^n$ be a Riemannian manifold of positive sectional curvature with quarter pinching. Then $N$ satisfies the condition $\sigma(3,\ell,\sup \mathrm{Ric}_{N}^{\ell-1};5/6)$ for every $3\leq \ell\leq n$. Moreover, when $3<\ell$, $N$ satisfies the strong condition $\sigma(3,\ell,\sup \mathrm{Ric}_{N}^{\ell-1};5/6)$.
\begin{proof}
For every orthonormal vectors $v_1,v_2,v_3$, consider $\tilde{K}_{ij}:=K_N(v_i,v_j)$. By symmetry, we have $\min_{J}\phi(A_{J}(r))=\min(\phi(A_{\{(1,2\}}(r)),\phi(A_{\{(1,2\}}(1/r)))$.
Let $A:=A_{ \{1,2\} }(r)\in M_{3\times 3}(\mathbb{R})$. Notice that $\min_{v\in \mathcal{S}^2} \phi_{A}(v)$ is attained by some $v^t=(x \ x \ 3-2x)$, and 
\[
    \phi:=\phi(A)=\phi_A(v)=\frac{(2r-2)x+6}{(2r-8)x^2+12x}.
\]
If $x=\frac{3}{2}$, then $\phi=\frac{2r+2}{3r}$. Now, suppose $0<x<\frac{3}{2}$. Then $A(\mathds{1}_{3}-2\phi v)\in \langle \mathds{1}_3\rangle$, and
\[
    \phi=\frac{r-1}{(2r-8)x+6}.
\]
Therefore, 
\[
    (r^2-5r+4)x^2+(6r-24)x+18=0,
\]
with discriminant $D=-36((r-1)^2-9)$ and solution $x=\frac{-(6r-24)-6\sqrt{(2+r)(4-r)}}{2(r-4)(r-1)}$, so $r\leq 4$. Evaluate $\phi$ at $x$ again. We get
\[
    \phi=\frac{(r-1)^2}{18-6\sqrt{(2+r)(4-r)}}.
\]
Thus, $\phi(A_{\{1,2\}}(4))=5/6$, and $\phi(A_{\{1,2\}}(1/4))=\frac{1}{2}+\frac{\sqrt{15}}{8}>5/6$, so it is proved by Proposition \ref{condition sigma: comparison to A_J(r)}.
\end{proof}
\end{corollary}

\section{Rigidity of Contracting Maps}

In this section, we build upon Theorems~\ref{first convergence result} and Theorem \ref{rigidity of contracting maps} to establish the rigidity of contracting maps.
The contracting conditions include upper bound for energy density, 2-nonnegativity of $\alpha=g-F^*h$ and general $k$-nonnegativity of $\alpha=\phi\cdot g-F^*h$. By rescaling the metric, these results can be extended to the setting where $\mathrm{Ric}_M > 0$.

The assumption about an upper bound on the energy density corresponds to condition $\sigma(\min(n,m),\min(n,m),c)$, valid for the space-form case.

\begin{corollary}\label{rigidity of contracting map between space forms}
Let $(M^m,g)$ and $(N^n,h)$ be closed connected manifolds. Assume their curvatures satisfy:
\begin{enumerate}[(a)]
    \item $\mathrm{Ric}_M\geq (\min(n,m)-1)\kappa$
    \item $K_N\leq \kappa$
\end{enumerate}
for some positive constant $\kappa$. Then for every smooth map $F\colon M\to N$ with $|d F|_{g,h}^2\leq \min(n,m)$, one of the following holds:
\begin{enumerate}[(i)]
    \item $F\colon M\to N$ is a Riemannian fiber bundle, and $K_N\equiv \kappa$;
    \item $F\colon M\to N$ is a totally geodesic isometric immersion, and $K_{M}\equiv \kappa$;
    \item $F$ is homotopically trivial.
\end{enumerate}
\begin{proof}
Since $N$ satisfies the strong condition $\sigma(\min(n,m),\min(n,m),(\min(n,m)-1)\kappa)$, the assumptions of Theorem \ref{rigidity of contracting maps} hold. We now focus on proving the rigidity of curvatures for cases (i) and (ii).

By Lemma \ref{characterizing limiting map}, $(K_{N})_{ij}\equiv \kappa$ for every $i\neq j$ with respect to any SVD frame. Hence, $K_{N}\equiv \kappa$ when $m\geq n$, and $\kappa\equiv (K_{N})|_{M}\equiv K_{M}$ when $m\leq n$.
\end{proof}
\end{corollary}

Since spheres are simply-connected, we have:
\begin{corollary}\label{Main theorem:sphere}
For any smooth map $F:(S^m.g_{S^m})\to (S^n,g_{S^n})$ with $|dF|^2\leq \min(n,m)$, $F$ is either an isometry embedding into a big-circle or homotopically trivial.
\end{corollary}

Theorem \ref{Einstein pinching} also follows from this theory directly:

\begin{proof}[Proof of Theorem \ref{Einstein pinching}]
By Corollary \ref{condition sigma for Einstein manifolds}, $N$ satisfies strong condition $\sigma(n,n,c)$, so the proof is done by Theorem \ref{rigidity of contracting maps}.
\end{proof}

For Einstein manifolds of positive sectional curvature, we have the following result:
\begin{corollary}
Let $(M^m,g)$ and $(N^n,h)$ be closed, connected manifolds. Suppose their curvatures satisfy the following conditions:
\begin{enumerate}[(a)]
\item $\mathrm{Ric}_M \geq \mathrm{Ric}_N$,
\item $N$ is an Einstein manifold of positive sectional curvature,
\item $(\sup K_N/\inf K_N)\leq r$ for some $r>2$.
\end{enumerate}
For a smooth map $F \colon M \to N$, if $|d F|^2 \leq \frac{n-2+r}{r-1}$, then $F$ is homotopically trivial.
\begin{proof}
By Proposition \ref{condition sigma for Einstein manifolds}, $N$ satisfies strong condition $\sigma(n,n,c;\frac{n-2+r}{nr-n})$, so the proof is done by Theorem \ref{rigidity of contracting maps}.
\end{proof}    
\end{corollary}

The general contracting properties and curvature pinching conditions also can be studied in this way:

\begin{proof}[Proof of Theorem \ref{k singular values, pinching}]
It's proved by Proposition \ref{condition sigma: comparison to A_J(r)} and Theorem \ref{rigidity of contracting maps}.
\end{proof}

\begin{proof}[Proof of Corollary \ref{Main theorem: complex projective spaces}]
It's proved by Corollary \ref{condition sigma for 1/4-pinching} and Theorem \ref{rigidity of contracting maps}.
\end{proof}

\section{Appendix}

\begin{proposition}\label{tensor maximum principle}
Let $M^m$ be a compact manifold with time-dependent metric $g(t)(t\in [0,T))$. Assume $h_{t}(t\in [0,T))$ is a time-dependent self-adjoint $(1,1)$-tensor on $M$ such that $h_{t}$ is $k$-nonnegative on $M\times \{0\}\cup \partial M\times [0,T)$ and that there are 
    \begin{enumerate}[(i)]
    \item a locally Lipschitz time-dependent bundle map $\phi:[0,T)\times \mathrm{Selfadj}(T^M)\to \mathrm{Selfadj}(T^M)$ such that $\sum_{i=1}^{k}g_{t}(\phi_t(\omega)(v_i),v_i)\geq 0$ if $\sum_{i=1}^{k}g_{t}(\phi_t(\omega)(v_i),v_i)=0$ for singular value decomposition $\{v_1,\ldots,v_m\}$ with singular values $\lambda_1\leq\cdots\leq \lambda_m$
    \item a vector field $X\in \mathfrak{X}(M)$ such that
    \[
        \frac{\partial}{\partial t}H=\Delta_{g(t)}H+\nabla_{X}H+\phi_t(H).
    \]
    \end{enumerate}
Then $h$ is $k$-nonnegative on $M\times [0,T)$.
\begin{proof}

First, observe that the set $\{t\in [0,T)\big|h_{s}\geq 0\ \forall 0\leq s\leq t\}$ is closed. If we can demonstrate its openness, we can conclude its connectivity. Therefore, our objective is to establish the existence of a small $\delta>0$ such that $h_{t}\geq 0$ for all $t\in [0,\delta)$. In this context, we can choose a space-time-uniform Lipschitz constant $K$ for $\phi$ (with respect to each metric $g(t)$)

For every $\epsilon>0$ and $\delta>0$, consider the $(1,1)$-tensor $\tilde{h}_{\delta,\epsilon}=\tilde{h}:=h+\epsilon(\delta+t)id$. Our strategy is to show that there exists a $\delta>0$ such that for every $\epsilon>0$, $\tilde{H}$ is $k$-nonnegative for all $t\in [0,\delta]$. Let $\delta$ be an undetermined positive constant.
Suppose not. Then there is an $\epsilon>0$ and a least $0<t_0\leq \delta$ such that $\tilde{h}_{t_0}$ is not $k$-positive at some $p_0\in M^{\circ}$. By continuity, $\tilde{h}_{t_0}$ is $k$-nonnegative, and
\[
    \sum_{i=1}^{k}g_{t_0}(h_{t_0}(v_i),v_i)=0,
\]
where $v_1,\ldots,v_k$ are the eigenvectors with smallest $k$ eigenvalues. By assumption, we have $\sum_{i=1}^{k}g_{t_0}(\phi_{t_0}(\tilde{h}_{t_0})(v_i),v_i)\geq 0$, and hence

    \begin{align*}
        \sum_{i=1}^{k}g_{t_0}(\phi_{t_0}(h_{t_0})(v_i),v_i)&\geq \sum_{i=1}^{k}g_{t_0}(\phi_{t_0}(\tilde{h}_{t_0})(v_i),v_i)-
        \sum_{i=1}^{k}|g_{t_0}(\phi_{t_0}(\tilde{h}_{t_0})(v_i),v_i)-g_{t_0}(\phi_{t_0}(h_{t_0})(v_i),v_i)|\\
        &\geq \sum_{i=1}^{k}g_{t_0}(\phi_{t_0}(\tilde{h}_{t_0})(v_i),v_i)-k|\phi_{t_0}(h_{t_0})-\phi_{t_0}(\tilde{h}_{t_0})|_{g(t_0)}\\
        &\geq -2kK\epsilon\delta.
    \end{align*}

On the other hand, consider SVD frame and the function
\[
    f(p,t):=\sum_{i=1}^{k}g_{t}(\tilde{h}_{t}(V_{i}),V_{i}).
\]
Then we have $f(p,t)>0$ for all $t<t_0$, and hence at $(p_0,t_0)$
    \begin{align*}
        0\geq \left(\frac{\partial }{\partial t}-\Delta\right)f&=k\epsilon+\sum_{i=1}^{k}g_{t_0}(\phi_{t_0}(h_{t_0})(v_i),v_i)\\
        &\geq k\epsilon-2kK\epsilon\delta=k\epsilon(1-2K\delta).
    \end{align*}
So, in this case, $\delta\geq \frac{1}{2K}$, and thus for every $\delta<\frac{1}{2M+2K}$, $\tilde{h}_{\delta,\epsilon}=h+\epsilon(\delta+t)g(t)$ is non-negative on $M\times [0,\delta]$. By taking $\epsilon\to 0^{+}$, we deduce that $h$ is non-negative on $M\times [0,\delta]$.
\end{proof}
\end{proposition}

\bibliographystyle{siamplain}
\bibliography{main}

\end{document}